\documentclass[english, 12pt]{amsart}

\usepackage{tikz}
\usepackage{aliascnt}
\usepackage{mathrsfs}
\usepackage[all,poly,knot]{xy}
\usepackage{hyperref}
\usepackage{comment}  
\usepackage{csquotes}   
\usepackage{amssymb,amsbsy,amsmath,amsfonts,amssymb,amscd,
	graphics,color,footmisc,fancyhdr,multicol,fancybox,
	graphicx,mathrsfs,rotating,ifthen,wasysym}

\usepackage{times}
\usepackage{pdfpages}
\usepackage{multirow}
\usepackage{nccrules}
\usepackage{textcomp}
\usepackage{comment}
\usepackage{framed}
\usepackage[all]{xy}
\usepackage{graphicx}
\usepackage{pgf,tikz}
\usepackage{mathrsfs}
\usetikzlibrary{arrows}
\usepackage{lipsum}
\usepackage{mathtools}

\newcommand{\medcup}{\mathbin{\scalebox{1.5}{\ensuremath{\cup}}}}

\DeclareMathOperator{\dif}{\text{\normalfont d}}

\DeclareMathOperator{\hol}{Hol}

\DeclareMathOperator{\ord}{ord}

\DeclareMathOperator{\bs}{Bs}

\def\log{\mathrm{log}\,}

\theoremstyle{plain}

\newtheorem{thm}{Theorem}[section]  

\newtheorem{cor}[thm]{{Corollary}} 

\newtheorem{lem}[thm]{{Lemma}}

\newtheorem{pro}[thm]{Proposition}

\newtheorem{rem}[thm]{Remark}
\newtheorem{defi}[thm]{Definition}

\theoremstyle{remark}

\numberwithin{equation}{section}

\usepackage[top=1.5in, bottom=1.5in, left=1in, right=1in]{geometry}

\usepackage[hyperpageref]{backref}

\usepackage{xcolor}
\hypersetup{
colorlinks,
linkcolor={red!50!black},
citecolor={blue!62!black},
urlcolor={blue!80!black}
}
\theoremstyle{plain}
\newcommand{\thistheoremname}{}
\newtheorem*{genericthm*}{\thistheoremname}
\newenvironment{namedthm*}[1]{\renewcommand{\thistheoremname}{#1}%
\begin{genericthm*}}
{\end{genericthm*}}

\newtheoremstyle{named}{}{}{\itshape}{}{\bfseries}{.}{.5em}{\thmnote{#3's }#1}
\theoremstyle{named}

\makeatletter
\newcommand\thankssymb[1]{\textsuperscript{\@fnsymbol{#1}}}
\makeatother
\begin{document} 
\title[A curvature estimate for holomophic maps]{\bf A curvature estimate for holomophic maps\\
	on open Riemann surfaces}


\author{Yunling Chen}
\address{Academy of Mathematics and Systems Sciences, Chinese Academy of Sciences, Beijing 100190, China}
\email{chenyl25@amss.ac.cn}

\author{Dinh Tuan Huynh}

\address{Department of Mathematics, University of Education, Hue University, 34 Le Loi St., Hue City, Vietnam}
\email{dinhtuanhuynh@hueuni.edu.vn}

\begin{abstract}
	We apply the technique of jet differentials to establish a Gauss curvature estimate for an open Riemann surface $M$, equipped with a conformal metric induced from a nonconstant holomorphic map that is highly ramified over a generic hypersurface of sufficiently high degree.
\end{abstract}

\maketitle

\section{Introduction}
Let $M$ be an open Riemann surface and $n\ge 3$. Let $f = (f_1, f_2,\ldots, f_n) \colon M\rightarrow \mathbb{R}^n$ be a conformal minimal immersion. The $(1, 0)$-differential $\partial f=(\partial f_1, \ldots, \partial f_n)$ of $f$ is a nonvanishing holomorphic map satisfying the equation
\[ \sum\limits_{i=1}^n \left(\partial f_i\right)^2=0 \quad \text{ everywhere on $M$}. \]
The \emph{generalized Gauss map} $g\colon M\rightarrow Q_{n-2}\subset\mathbb{CP}^{n-1}$ of the conformal minimal immersion $f$ is then defined as
$$
g(z):=\left[\frac{\partial f_1}{\partial z}: \cdots: \frac{\partial f_n}{\partial z}\right],
$$
where $Q_{n-2}=\left\{[z_1:\ldots:z_n]\in \mathbb{CP}^{n-1}\, \bigg|\, \sum\limits_{i=1}^n z_i^2=0 \right\}$ is the hyperquadric in $\mathbb{CP}^{n-1}$ and $z=x+iy$ is a holomorphic chart on $M$. Abbreviated as \emph{the Gauss map} of the conformal minimal immersion, $g$ is indeed holomorphic \cite{Fujimoto83_1}. It has enabled to exploit complex-analytic methods to study minimal surfaces in Euclidean space, leading to many fundamental developments in the theory.

In the low-dimensional case where $M\hookrightarrow \mathbb{R}^3$, by identifying the projective conic $Q_1$ with the Riemann sphere $\mathbb{CP}^{1}$, one may regard the Gauss map $g$ as a meromorphic function on $M$. When $M$ is non-flat and complete,  many analogs of the Little Picard Theorem for $g$ were established. For example, Osserman \cite{Osserman64} proved that  the complement $\mathbb{C}\mathbb{P}^1\setminus g(M)$ is of logarithmic capacity zero in $\mathbb{C}\mathbb{P}^1$. Xavier \cite{Xavier81} improved this result by proving that this complement can contain  at most $6$ points. Sharp result was obtained by Fujimoto \cite{Fujimoto88}, who demonstrated that $g$ can avoid at most $4$ points. Furthermore, Fujimoto \cite{Fujimoto83_2} introduced the notion of {\sl non-integrated defect} for the Gauss map and established a {\sl non-integrated defect relation}, which is an analog of the classical defect relation in Nevanlinna theory and quantifies the previous Little Picard-type theorems.

Related with these results, Fujimoto \cite{Fujimoto88} proved that if the Gauss map $g$, of a conformal minimal immersion in $M\hookrightarrow \mathbb{R}^3$, omits at least 5 distinct points, then there exists a constant $C$, depending only on the omitted points, such that
\begin{equation}
	\label{curvature estimate, Fujimoto}
	|K(q)|^{\frac{1}{2}}\mathrm{d}(q)\leq C,
\end{equation}
where $K(q)$ is the Gauss curvature of $M$ at $q$ and $\mathrm{d}(q)$ is the geodesic distance from $q$ to the boundary of $M$.

Motivated by the value distribution theory for holomorphic curves, many generalizations concerning the Gauss map of conformal minimal immersion in higher dimensions have been established. Notably, Fujimoto \cite{Fujimoto90} proved that if the generalized Gauss map of a complete minimal surface $M\hookrightarrow\mathbb{R}^{n}$ omits a family of $q>\frac{n(n+1)}{2}$ hyperplanes in general position in $\mathbb{CP}^{n-1}$, then the map must be linearly degenerate. Subsequently, Ru \cite{Ru91} strengthened this result by proving that under these conditions, such generalized Gauss maps must be constant. In the quantitative direction, several higher-dimensional non-integrated defect relations were established in \cite{Fujimoto89, Fujimoto90, Fujimoto91}. Moreover, Osserman-Ru \cite{Osserman-Ru97} obtained the curvature estimate \eqref{curvature estimate, Fujimoto} for the generalized Gauss maps omitting more than $\frac{n(n+1)}{2}$ hyperplanes located in general position in $\mathbb{CP}^{n-1}$, thereby extending Fujimoto's result to higher dimensions. Recently, Chen-Li-Liu-Ru \cite{Chen-Li-Liu-Ru21} extended the curvature estimate result with the \emph{weighted conformal metric} induced from the holomorphic map $g\colon M\to \mathbb{CP}^{n-1}$. Shortly afterward, Si \cite{Quang2023} generalized their result under a weaker assumption that the Gauss map is ramified over a family of hypersurfaces in general position with sufficiently high multiplicities.

A key tool in the study of the Gauss map stems from the Second Main Theorem in value distribution theory. Generalizing the Little Picard Theorem to higher dimensions, the Kobayashi conjecture posits the hyperbolicity of a generic high-degree hypersurface in projective space, as well as of its complement. In recent decades, significant progress has been made toward this conjecture \cite{Berczi-Kirwan24, Brotbek17, Brotbek-Deng19, cadorel2024, Darondeau2016, Demailly20, DMR2010, Riedl-Yang22, Siu2015}, which also extends the scope of Second Main Theorem \cite{Huynh-Vu-Xie-2017}. Therefore, it is  natural to seek analogues of hyperbolicity results for the Gauss maps of a conformal minimal immersion of an open Riemann surface. Extending the main results of \cite{Huynh-Vu-Xie-2017} to holomorphic maps from a bounded disc in $\mathbb{C}$, the second-named author obtained the constancy of generalized Gauss map of a complete minimal surface $M\hookrightarrow\mathbb{R}^n$, omitting a generic hypersurface in $\mathbb{CP}^{n-1}$ of large enough degree. Subsequently, a quantitative result in term of a non-integrated defect relation for holomorphic maps into projective varieties was established in \cite{Cai-Ru-Yang25}.

In this paper, we aim to give a curvature estimate for a holomorphic map from an open Riemann surface $M$ to $\mathbb{CP}^n$, ramified over a generic hypersurface of large degree with sufficiently high multiplicity. First, we prove that the Gauss curvature estimate \eqref{curvature estimate, Fujimoto} can be deduced from the existence of a suitable logarithmic jet differential.

\begin{thm}
	\label{gauss curvature estimate}
	Let $M$ be an open Riemann surface and let $f\colon M\rightarrow\mathbb{CP}^n$ be a non-constant holomorphic map. For a positive integer $p$, consider the conformal metric on $M$ given by
	\[
	\mathrm{d}s^2=\|F\|^{2p}|\omega|^2,
	\]
	where $F$ is a reduced representation of $f$ and $\omega$ is a holomorphic $1$--form on $M$. Let $D\subset\mathbb{CP}^n$ be a generic Kobayashi hyperbolic hypersurface of degree $d$ such that $f(M)\not\subset D$.
	Suppose that for two positive integers $m,\tilde{m}$ with $\tilde{m}>2pm$, the subset
	\[
	\bs	\big(
	E_{n,m}T_{\mathbb{C}\mathbb{P}^n}^*(\log D)
	\otimes
	\mathcal{O}_{\mathbb{C}\mathbb{P}^n}(1)^{-\widetilde{m}}
	\big)\cap \left(\mathbb{C}\mathbb{P}^n\setminus D\right)
	\]
	has dimension at most zero. Assume also that $f$ is ramified over $D$ with multiplicity at least $\mu$ with
	\[
	\dfrac{1}{\mu}<\dfrac{1}{d}\bigg(\dfrac{\widetilde{m}}{m}-2p\bigg).
	\] 
	Then there exists a positive constant $C$ depending only on $D$ such that
	\begin{equation}\label{eq:curvature-main}
		|K(q)|^{\frac{1}{2}}\mathrm{d}(q)\leq C,
	\end{equation}
	where $K(q)$ is the Gauss curvature of $M$ at $q$ with respect to the metric $\mathrm{d}s^2$ and $\mathrm{d}(q)$ is the geodesic distance from $q$ to the boundary of $M$ in the same metric.
\end{thm}

Based on the recent progress about the Green-Griffiths jet bundle, we show that the assumption of Theorem \ref{gauss curvature estimate} is satisfied for a generic hypersurface of large degree.

\begin{cor}
	\label{cor of gauss curvature estimate}
	Let $M$ be an open Riemann surface and let $f\colon M\rightarrow\mathbb{CP}^n$ be a non-constant holomorphic map. For a positive integer $p$, consider the conformal metric on $M$ given by
	\[
	\mathrm{d}s^2=\|F\|^{2p}|\omega|^2,
	\]
	where $F$ is a reduced representation of $f$ and $\omega$ is a holomorphic $1$--form on $M$. 	Let $D\subset\mathbb{CP}^n$ be a generic hypersurface having degree
	\begin{equation*}
		d\geq d_{n,p}=
		\begin{cases}
			16n(p + 5n -2)\left(\sqrt{2n}\, \frac{\log\log(2n)}{2}+3\right)^{2n},\quad \text{for }\,2\leq n\leq 39;\\
			\left(1+\frac{2p-5}{10n+1}\right)\left(n\log n\right)^n,\quad \text{whenever\, $n\geq 40$}
		\end{cases}
	\end{equation*}
	such that $f(M)\not\subset D$. 
	Suppose that $f$ is ramified over $D$ with multiplicity at least $\mu$, where ${\mu}>{d}.$ 
	Then there exists a positive constant $C$ depending only on $D$ such that, for each $q\in M$, the curvature estimate \eqref{eq:curvature-main} holds.
\end{cor}

We now outline some key ideas in the proof of Theorem~\ref{gauss curvature estimate}. Our aim is to establish the curvature estimate \eqref{eq:curvature-main} without assuming the completeness of the metric $\mathrm{d}s^2$. To this end, following the approach in \cite{Chen-Li-Liu-Ru21} and \cite{Quang2023}, we need two main ingredients derived from suitable logarithmic jet differentials along a generic hypersurface $D$: a non-integrated defect relation for nonconstant holomorphic maps from an open Riemann surface with respect to $D$, and a normality criterion for families of holomorphic maps from the unit disc that are ramified over $D$. The first result, concerning the non-integrated defect relation (see Theorem~\ref{defect gauss map}), is obtained by adapting similar arguments as in the proof of \cite[Theorem 1.5]{Cai-Ru-Yang25}. This extends Fujimoto's technique \cite{Fujimoto83_2} from the special jet differential built upon Wronskian  to the general logarithmic jet differentials. The second result, a normality criterion (see Theorem~\ref{normal family}), follows from a variant of the Second Main Theorem for logarithmic jet differentials established in \cite{Huynh-Vu-Xie-2017}, where the source of the holomorphic maps is taken to be the unit disc. Finally, with the above two results in hand, Theorem~\ref{gauss curvature estimate} is proved by following similar arguments as in \cite{Chen-Li-Liu-Ru21} (see also \cite{Quang2023}).

\section{Second Main Theorem from logarithmic jet differentials}
\subsection{Logarithmic jet differentials}
Let $X$ be a complex projective variety of dimension $n$. For a point $x\in X$, consider the holomorphic germs $(\mathbb{C},0)\rightarrow (X,x)$. Two such germs are said to be equivalent if they have the same Taylor expansion up to order $k$ in some local coordinates around $x$. The equivalence class of an analytic germ $f\colon (\mathbb{C},0)\rightarrow (X,x)$ is called the {\sl $k$-jet of $f$}, denoted by $j_k(f)$, which is independent of the choice of local coordinates. A $k$-jet $j_k(f)$ is said to be {\sl regular} if $\dif f(0)\not=0$. For a given point $x\in X$, denote by $j_k(X)_x$ the vector space of all $k$-jets of analytic germs $(\mathbb{C},0)\rightarrow (X,x)$, set
\[
J_k(X)
:=
\underset{x\in X}{\medcup}\,J_k(X)_x,
\]
and consider the natural projection
\[
\pi_k\colon J_k(X)\rightarrow X.
\]
Then $J_k(X)$ carries the structure of a holomorphic fiber bundle over $X$, which is called the {\sl $k$-jet bundle over $X$}. In general, $J_k(X)$ is not a vector bundle. When $k=1$, the $1$-jet bundle $J_1(X)$ is canonically isomorphic to the tangent bundle $T_X$.

For an open subset $U\subset X$, a section $\omega\in H^0(U,T_X^*)$, and a $k$-jet $j_k(f)\in J_k(X)|_U$, the pullback $f^*\omega$ is of the form $A(z)\dif z$ for some analytic function $A$ on $U$. Since each derivative $A^{(j)}$ ($0\leq j\leq k-1$) is well-defined, independent of the representation of $f$ in the class $j_k(f)$, the analytic $1$-form $\omega$ induces the holomorphic map
\begin{align}
	\label{trivialization-jet}
	\tilde{\omega}
	\colon
	J_k(X)|_U &
	\rightarrow
	\mathbb{C}^k;\nonumber\\
	j_k(f)(z)&\mapsto
	\big(A(z),A(z)^{(1)},\dots,A(z)^{(k-1)}
	\big).
\end{align}
Hence on an open subset $U$, a local holomorphic coframe $\omega_1\wedge\dots\wedge\omega_n\not=0$ yields a trivialization 
\[
H^0(U, J_k(X))\rightarrow U\times(\mathbb{C}^k)^n,
\]
with the following new $n k$ independent coordinates:
\[
\sigma\mapsto(\pi_k\circ\sigma;\tilde{\omega}_1\circ\sigma,\dots,\tilde{\omega}_n\circ\sigma),
\]
where $\tilde{\omega}_i$ are defined as in \eqref{trivialization-jet}. The components $x_i^{(j)}$ ($1\leq i\leq n$, $1\leq j\leq k$) of $\tilde{\omega}_i\circ\sigma$ are called the {\sl jet-coordinates}.
In a more general setting, where $\omega$ is a section over $U$ of the sheaf of meromorphic $1$-forms, the induced map $\tilde{\omega}$ is meromorphic.

Now, suppose that $D\subset X$ is a normal crossing divisor on $X$. This means that at each point $x\in X$, there exist some local coordinates $z_1,\dots,z_{\ell},z_{\ell+1},\dots,z_n$ ($\ell=\ell(x)$) centered at $x$, such that $D$ is defined by
\[
D=
\{z_1\dots z_{\ell}=0
\}.
\]
Following Iitaka \cite{Itaka82}, the {\sl logarithmic cotangent bundle of $X$ along $D$}, denoted by $T_X^*(\log D)$, corresponds to the locally free sheaf generated by
\[
\dfrac{\dif\!z_1}{z_1},\dots,\dfrac{\dif\! z_{\ell}}{z_{\ell}},z_{\ell +1},\dots,z_n.
\]

A holomorphic section $s\in H^0(U,J_k(X))$ over an open subset $U\subset X$ is said to be a {\sl logarithmic $k$-jet field} if $\tilde{\omega}\circ s$ is analytic for any section $\omega\in H^0(U',T_X^*(\log D))$ and any open subsets $U'\subset U$. Such logarithmic $k$-jet fields define a subsheaf of $J_k(X)$, which is a sheaf of sections of a holomorphic fiber bundle over $X$. Such bundle is called the {\sl logarithmic $k$-jet bundle over $X$ along $D$}, denoted by $J_k(X,-\log D)$ (see \cite{Noguchi1986}).


A {\sl logarithmic jet differential of {\sl order} $k$ and {\sl degree}} $m$ at a point $x\in X$ is a polynomial $Q(f^{(1)},\dots,f^{(k)})$ on the fiber over $x$ of $J_k(X,-\log D)$ enjoying the weighted homogeneity:
\[
Q(j_k(f\circ{\lambda}))
=
\lambda^m
Q(j_k(f))
\eqno\scriptstyle{(\lambda\,\in\,\mathbb{C}^*)}.
\]
Consider the symbols
\begin{align*}
	\dif^{j}\log z_i\, {\scriptstyle{(1\,\leq\, j\,\leq\, k,\,1\,\leq\, i\,\leq\,\ell)}} \quad\text{and}\quad {\dif^{j}z_i\, {\scriptstyle{(1\,\leq\, j\,\leq\, k,\,\ell\,+\,1\,\leq\, i\,\leq\,n)}}}.  
\end{align*}
Set the weight of $\dif^{j}\log z_i$ or $\dif^{j}z_i$ to be $j$. Then
a logarithmic jet differential of order $k$ and weight $m$ along $D$ at $x$ is a weighted homogeneous polynomial of degree $m$ in the variables given by the symbols above. Denote by $E_{k,m}^{GG}T_X^*(\log D)_x$ the vector space spanned by such polynomials and set
\[
E_{k,m}^{GG}T_X^*(\log D)
:=
\underset{x\in X}{\medcup}\,
E_{k,m}^{GG}T_X^*(\log D)_x.
\]
By Fa\`{a} di bruno's formula \cite{Constantine1996,Merker2015}, one can check that $E_{k,m}^{GG}T_X^*(\log D)$ carries the structure of a vector bundle over $X$, called {\sl logarithmic Green-Griffiths vector bundle} \cite{Green-Griffiths1980}. A global section of $E_{k,m}^{GG}T_X^*(\log D)$ is called a {\sl logarithmic jet differential} of order $k$ and weight $m$ along $D$. Locally, a logarithmic jet differential form can be written as

{\footnotesize
	\begin{equation}
		\label{local expression of log jet}
		\underset{|\alpha_1|+2|\alpha_2|+\dots+k|\alpha_k|=m}
		{\sum_{\alpha_1,\dots,\alpha_k\in\mathbb{N}^n}}
		A_{\alpha_1,\dots,\alpha_k}
		\bigg(
		\prod_{i=1}^{\ell}
		\big(
		\dif\log z_i
		\big)^{\alpha_{1,i}}
		\prod_{i=\ell+1}^{n}
		\big(
		\dif z_i
		\big)
		^{\alpha_{1,i}}
		\bigg)
		\dots
		\bigg(
		\prod_{i=1}^{\ell}
		\big(
		\dif^k\log z_i
		\big)^{\alpha_{k,i}}
		\prod_{i=\ell+1}^{n}
		\big(
		\dif^kz_i
		\big)
		^{\alpha_{k,i}}
		\bigg),
	\end{equation}
}
where $ \alpha_{\lambda}
=
(\alpha_{\lambda,1},\dots,\alpha_{\lambda,n})
\in\mathbb{N}^n
{\scriptstyle{(1\,\leq\,\lambda\,\leq\, k)}}$
is the multi-indices of length $|\alpha_{\lambda}|
=
\sum_{1\leq i\leq n}
\alpha_{\lambda,i},$
and where $A_{\alpha_1,\dots,\alpha_k}$ are locally defined holomorphic functions. 

Assigning the weight $s$ for $\frac{\dif^s z_i}{z_i}$, one can rewrite $\dif^j\log z_i$ as an isobaric polynomial of variables $\frac{\dif^s z_i}{z_i}\, (1\leq s\leq j)$ with integer coefficients, namely
$$
\dif^j\log z_i
=
\underset{\beta_1+2\beta_2+\dots+j\beta_j=j}
{\sum_{\beta=(\beta_1,\dots,\beta_j)\in\mathbb{N}^j}}
b_{j\beta}\bigg(\dfrac{\dif z_i}{z_i}\bigg)^{\beta_1}\dots\bigg(\dfrac{\dif^j z_i}{z_i}\bigg)^{\beta_j},
$$
where $b_{j\beta}\in\mathbb{Z}$. Conversely, one can also express $\frac{\dif^j z_i}{z_i}$ as an isobaric polynomial of weight $j$ of variables $\dif^s\log z_i$ ($1\leq s\leq j$) with integer coefficients \cite{Brotbek-Deng19}. Thus one can also use the following trivialization of logarithmic jet differentials:
{\footnotesize
	\begin{equation}
		\label{local expression of log jet, second form}
		\underset{|\beta_1|+2|\beta_2|+\dots+k|\beta_k|=m}
		{\sum_{\beta_1,\dots,\beta_k\in\mathbb{N}^n}}
		B_{\beta_1,\dots,\beta_k}
		\bigg(
		\prod_{i=1}^{\ell}
		\big(
		\frac{\dif z_i}{z_i}\big)^{\beta_{1,i}}
		\prod_{i=\ell+1}^{n}
		\big(
		\dif z_i
		\big)
		^{\beta_{1,i}}
		\bigg)
		\dots
		\bigg(
		\prod_{i=1}^{\ell}
		\big(\dfrac{
			\dif^k z_i}{z_i}
		\big)^{\beta_{k,i}}
		\prod_{i=\ell+1}^{n}
		\big(
		\dif^kz_i
		\big)
		^{\beta_{k,i}}
		\bigg),
	\end{equation}
}
where $\beta_{\lambda}
=
(\beta_{\lambda,1},\dots,\beta_{\lambda,n})
\in\mathbb{N}^n
{\scriptstyle{(1\,\leq\,\lambda\,\leq\, k)}}$
is multi-indices of length 
$
|\beta_{\lambda}|
=
\sum_{1\leq i\leq n}
\beta_{\lambda,i},
$
and where $B_{\beta_1,\dots,\beta_k}$ are locally defined holomorphic functions. 

Demailly \cite{Demailly97} refined the Green-Griffiths' theory and considered the sub-bundle $E_{k,m}T^*_X(\log D)$, whose sections are logarithmic jet differentials that are invariant under arbitrary reparametrization of the source $\mathbb{C}$. More precisely, for any $\phi\colon \left(\mathbb{C},0\right)\to \left(\mathbb{C},0\right)$, a germ of $k$-jets biholomorphisms, the fibers of  $E_{k,m}T^*_X(\log D)$ satisfies that
\[Q(j_k(f\circ{\phi}))
=
\left(\phi^\prime (0)\right)^m
Q(j_k(f))
.\]

The bundles $E_{k,m}^{GG}T_X^*(\log D)$ and $E_{k,m}T_X^*(\log D)$ are fundamental tools in studying (logarithmic) hyperbolicity conjecture. By the fundamental vanishing theorem \cite{Demailly97, Siu2015}, for any ample line bundle $\mathcal{A}$ on $X$, a non-trivial global section of $E_{k,m}^{GG}T_X^*(\log D)\otimes\mathcal{A}^{-1}$ (or $E_{k,m}T_X^*(\log D)\otimes\mathcal{A}^{-1}$) gives a corresponding algebraic differential equation that all entire curve $f\colon\mathbb{C}\rightarrow X\setminus D$ must satisfy. That is, any entire curve $f\colon\mathbb{C}\rightarrow X\setminus D$ is contained in the base locus of $E_{k,m}^{GG}T_X^*(\log D)\otimes\mathcal{A}^{-1}$ (or $E_{k,m}T_X^*(\log D)\otimes\mathcal{A}^{-1}$). However, despite many efforts, controlling the base locus remains feasible only when $X=\mathbb{CP}^n$ and the degree of hypersurface $D$ must be sufficiently large (see \cite{Berczi-Kirwan24, Brotbek-Deng19, Demailly20, DMR2010, Siu2015}).
\subsection{Second Main Theorem} 
As Bloch’s famous dictum ``{\it Nothing exists in the infinite plane that has not been previously done in the finite disc}''  \cite{Bloch26}, we study the Second Main Theorem for holomorphic maps from a bounded disc $\Delta_R:=\{z\in\mathbb{C}:|z|<R\}$, with $0<R< \infty$. 
Let $E=\sum\limits_i\alpha_i\,a_i$ be an effective divisor in $\Delta_R$, where $\alpha_i\geq 0$ and $a_i\in\Delta_R$ for all $i$. Let $k\in \mathbb{N}\cup\{\infty\}$. For any $0<t<R$, we denote by
\[
n^{[k]}(t,E)
:=
\sum_{a_i\in\Delta_t\cap E}
\min
\,
\{k,\alpha_i\},
\]
the $k$-truncated degree of $E$ on the disc $\Delta_t$.
Then for $0\,<\,r\,<R$, the \textsl{truncated counting function at level} $k$ of $E$ is defined via the logarithmic average:
\[
N^{[k]}(r,E)
\,
:=
\,
\int_0^r \frac{n^{[k]}(t, E)}{t}\,\dif\! t.
\]
When $k=\infty$, we omit the superscript and write simply $n(t,E)$ and $N(r,E)$ in place of $n^{[\infty]}(t,E)$ and $N^{[\infty]}(r,E)$. 

Let $f\colon\Delta_R\rightarrow \mathbb{C}\mathbb{P}^n$ be a holomorphic disc, having a reduced representation $f=[f_0:\cdots:f_n]$ in the homogeneous coordinates $[z_0:\cdots:z_n]$ of $\mathbb{C}\mathbb{P}^n$. Let $D=\{Q=0\}$ be a divisor in $\mathbb{C}\mathbb{P}^n$ defined by a homogeneous polynomial $Q\in\mathbb{C}[z_0,\dots,z_n]$ of degree $d\geq 1$. If $f(\Delta_R)\not\subset D$, we define the \textsl{$k$-truncated counting function} of $f$ with respect to $D$ as
\[
N_f^{[k]}(r,D)
\,
:=
\,
N^{[k]}\big(r,(Q\circ f)_0\big),
\]
where $(Q\circ f)_0$ denotes the zero divisor of $Q\circ f$ in $\Delta_R$.

The \textsl{proximity function} of $f$ for the divisor $D$ is defined as
\[
m_f(r,D)
\,
:=
\,
\int_0^{2\pi}
\log
\frac{\big\Vert f(re^{i\theta})\big\Vert^d\,
	\Vert Q\Vert}{\big|Q(f)(re^{i\theta})\big|}
\,
\frac{\dif\!\theta}{2\pi},
\]
where $\Vert Q\Vert$ is the maximum  absolute value of the coefficients of $Q$ and
\[
\big\Vert f(z)\big\Vert
\,
\coloneq
\,
\max
\,
\{|f_0(z)|,\ldots,|f_n(z)|\}.
\]
Since the Cauchy's inequality $\big|Q(f)\big|\leq \Vert Q\Vert\cdot\Vert f\Vert^d$, one has $m_f(r,D)\geq 0$.

Lastly, the \textsl{Cartan order function} of $f$ is defined by
\begin{align*}
	T_f(r)
	\,
	:&=
	\,
	\frac{1}{2\pi}\int_0^{2\pi}
	\log
	\big\Vert f(re^{i\theta})\big\Vert \dif\!\theta.
\end{align*}

With above notations, Nevanlinna theory consists of the following two fundamental theorems (for comprehensive presentations, see  \cite{Noguchi-Winkelmann2014,Ru21}).

\begin{thm}[First Main Theorem]\label{fmt} Let $f\colon\Delta_R\rightarrow \mathbb{C}\mathbb{P}^n$ be a holomorphic curve and let $D\subset \mathbb{C}\mathbb{P}^n$ be a hypersurface of degree $d$ such that $f(\Delta_R)\not\subset D$. Then for every $0<r<R$, one has
	\[
	m_f(r,D)
	+
	N_f(r,D)
	\,
	=
	\,
	d\,T_f(r)
	+
	O(1).
	\]
	Consequently, one obtains the Nevanlinna's inequality:
	\begin{equation}
		\label{-fmt-inequality}
		N_f(r,D)
		\,
		\leq
		\,
		d\,T_f(r)+O(1).
	\end{equation}
\end{thm}
For $\mu_0\in \mathbb{N}\cup\{\infty\}$, the usual $\mu_0$-truncated defect of $f$ with respect to $D$ is given by 
\begin{equation}\label{def-truncated (usual) defect}
	\delta_{f,D}^{[\mu_0]}=\liminf_{r\rightarrow\infty}
	\bigg(
	1-
	\dfrac{N^{[\mu_0]}_f(r,D)}{d\,T_f(r)}
	\bigg).
\end{equation}
Hence, it follows directly from the First Main Theorem that
$
0 \leq \delta_{f,D}^{[\mu_0]}\leq 1.
$

Conversely, at a more profound level, the so-called Second Main Theorem seeks to bound the order function from above by a sum of appropriate counting functions. Results of this type have been established in several special cases, and most rely on the following key estimate.

\begin{thm}[Logarithmic Derivative Lemma \cite{Ru2023}]\label{lem:ldl-disc-r}
	Let $g\colon \Delta_R\rightarrow \mathbb{C}\mathbb{P}^1$ be a non-constant meromorphic function and let $k\geq 1$ be a positive integer. Then for any $0<r<R$, the following estimate holds:
	\[
	m_{\frac{g^{(k)}}{g}}
	(
	r
	)
	:=
	m_{\frac{g^{(k)}}{g}}
	(
	r,\infty
	)
	=
	O\bigg(\log\dfrac{1}{R-r}\bigg)+O(\log T_g(r))\qquad\parallel,
	\]
	where the notation $\parallel$ means that the estimate holds true for all $0<r<R$ outside a subset $E\subset (0,R)$ with finite measure $\int_E\dfrac{\mathrm{d}r}{R-r}<\infty.$
\end{thm}
\begin{thm}[Second Main Theorem] \label{smt from jet}
	Let $D\subset\mathbb{C}\mathbb{P}^n$ be a hypersurface of degree $d$.
	Let $f\colon\Delta_R\rightarrow \mathbb{C}\mathbb{P}^n$ be a non-constant holomorphic disc such that $f(\Delta_R)\not\subset D$. Suppose that
	\[
	\mathscr{P}
	\in
	H^0
	\big(
	\mathbb{C}\mathbb{P}^n,
	E_{k,m}T_{\mathbb{C}\mathbb{P}^n}^*(\log D)
	\otimes
	\mathcal{O}_{\mathbb{C}\mathbb{P}^n}(1)^{-\widetilde{m}}
	\big)
	\]
	is a nontrivial negatively twisted logarithmic $k$-jet differential such that
	\begin{align}
		\label{eq_canPfjets-section2}
		\mathscr{P}\big(j_k(f)
		\big)
		\not\equiv
		0,
	\end{align}
	then we have
	\begin{align}\label{eq:smt-jet-disc-r}
		T_f(r)\leq\dfrac{m}{\widetilde{m}}N^{[1]}_f(r,D)+o(\log (T_f(r))+O\bigg(\log\dfrac{1}{R-r}\bigg)\quad\parallel.
	\end{align}
\end{thm}
\begin{proof}
	This is a direct consequence from Theorem \ref{lem:ldl-disc-r} and similar techniques in \cite[Theorem 3.1]{Huynh-Vu-Xie-2017}.
\end{proof}
\begin{cor}[Defect relation]
	\label{defect relation}
	Keep the same assumptions as in Theorem~\ref{smt from jet}. If either $R=\infty$, or $R<\infty$ and $\limsup\limits_{r\rightarrow R}\frac{T_f(r)}{-\log (R-r)}=\infty$, then one has the 1-truncated defect relation as
	\[
	0\leq \delta^{[1]}_{f,D}\leq 1-\dfrac{\widetilde{m}}{md}.
	\]
\end{cor}
\begin{proof}
	When $R=\infty$, the Second Main Theorem from \cite[Theorem 3.1]{Huynh-Vu-Xie-2017} yields that
	\begin{equation}\label{eq:smt-HVX}
		T_f(r)\leq\dfrac{m}{\widetilde{m}}N^{[1]}_f(r,D)+o(\log (T_f(r)).
	\end{equation}
	On the other hand, when $R<\infty$ and $\limsup\limits_{r\rightarrow R}\frac{T_f(r)}{-\log (R-r)}=\infty$, the inequality \eqref{eq:smt-jet-disc-r} holds. By the definition of the truncated defect \eqref{def-truncated (usual) defect}, we obtain the stated result.
\end{proof}

\begin{cor}[Ramification estimate]
	\label{ramification}
	Keep the same assumption as in Theorem~\ref{smt from jet} and suppose that $f$ is ramified over $D$ with multiplicity at least $\mu$. If either $R=\infty$, or $R<\infty$ and $\limsup_{r\rightarrow R}\frac{T_f(r)}{-\log (R-r)}=\infty$, then one has the ramification estimate
	\[
	\mu\leq \dfrac{md}{\widetilde{m}}.
	\]
\end{cor}
\begin{proof}
	If $f$ is ramified over $D$ with multiplicity at least $\mu$, then the Nevanlinna's inequality \eqref{-fmt-inequality} yields
	\[N^{[1]}_f(r,D)\leq \frac{1}{\mu} N_f(r,D)\leq \frac{d}{\mu}T_f(r)+O(1).\]
	Combining with \eqref{eq:smt-jet-disc-r} and \eqref{eq:smt-HVX}, we derive the stated bound.
\end{proof}

\subsection{Existence of logarithmic jet differentials with high vanishing order}
To guarantee the existence of the logarithmic jet differentials satisfying \eqref{eq_canPfjets-section2}, we employ the results in \cite{Merker-Ta22}, improving \cite{Darondeau2016} together with \cite{Riedl-Yang22}.

\begin{cor}
	\label{cor of ramification gauss maps}
	Let $D\subset\mathbb{C}\mathbb{P}^n$ be a generic hypersurface of degree $d$. For an arbitrary positive integer $p$, if the degree $d$ of $D$ is large enough, namely
	\begin{equation}\label{eq:degree bound d }
		d\geq d_{n,p}=
		\begin{cases}
			16n(p + 5n -2)\left(\sqrt{2n}\, \frac{\log\log(2n)}{2}+3\right)^{2n},\quad \text{for }\,2\leq n\leq 39;\\
			\left(1+\frac{2p-5}{10n+1}\right)\left(n\log n\right)^n,\qquad\qquad\qquad\qquad\, \text{whenever\, $n\geq 40$}.
		\end{cases}
	\end{equation}
	Then there exists large integers $m\gg1$ and $\tilde{m}\geq(2p+1)m>2pm$, such that the subset
	\[ B\coloneq
	\bs	\big(
	E_{n,m}T_{\mathbb{C}\mathbb{P}^n}^*(\log D)
	\otimes
	\mathcal{O}_{\mathbb{C}\mathbb{P}^n}(1)^{-\widetilde{m}}
	\big)\cap \left(\mathbb{C}\mathbb{P}^n\setminus D\right),
	\]
	where \( \bs(\cdot) \) denotes the base locus,
	has at most zero dimension. 
	
	Furthermore, let $f\colon\Delta\rightarrow \mathbb{C}\mathbb{P}^n$ be a non-constant holomorphic disc such that $f(\Delta)\not\subset D$. Then there exists a nontrivial negatively twisted logarithmic $n$-jet differential
	\[
	\mathscr{P}\in H^0\big(
	\mathbb{C}\mathbb{P}^n,E_{n,m}T_{\mathbb{C}\mathbb{P}^n}^*(\log D)\otimes\mathcal{O}_{\mathbb{C}\mathbb{P}^n}(1)^{-\widetilde{m}}\big)
	\]
	with above $m,\widetilde{m}$, satisfying that
	\begin{align*}
		\mathscr{P}\big(j_k(f)\big)\not\equiv0.
	\end{align*}
\end{cor}
\begin{proof}
	Let $U_{r,d}$ denote the space of smooth hypersurfaces of degree $d$ in $\mathbb{C}\mathbb{P}^r$ for arbitrary positive integer $r$, and consider the ambient space $\mathbb{C}\mathbb{P}^r\times U_{r,d}$. Define the universal hypersurface $$\mathcal{X}_{r,d}=\left\{(p,X)\in\mathbb{C}\mathbb{P}^r\times U_{r,d}\, \big|\, p\in X\right\}.$$ 
	Let $\pi_2\colon \mathbb{C}\mathbb{P}^r\times U_{r,d}\to U_{r,d}$ be the canonical projection and let $T_{\pi_2}$ be the relative tangent space. For positive integers $k,m,\widetilde{m}$, we set
	\[\bs_{k,m,\widetilde{m}}=\bs\left(H^0\left(\mathbb{C}\mathbb{P}^r\times U_{r,d},E_{k,m}T_{\pi_2}^*\left(\log \mathcal{X}_{r,d}\right)\otimes\mathcal{O}_{\mathbb{C}\mathbb{P}^r}(-\widetilde{m})\right)\right).\]
	
	Let $Z_{r,d}\subset \left(\mathbb{C}\mathbb{P}^r\times U_{r,d}\right)\setminus\mathcal{X}_{r,d}$ be the locus of pairs $(p, X)$ where $p \in (\mathbb{C}\mathbb{P}^r\setminus X)\cap \bs_{k,m,\widetilde{m}}$.
	
	According to \cite[Theorem 3.6]{Riedl-Yang22}, if for some pair \( (n, d_{n}) \), the locus $Z_{n,d_{n}}$ has codimension at least 1 in $ \left(\mathbb{C}\mathbb{P}^{n}\times U_{n,d_{n}}\right)\setminus\mathcal{X}_{n,d_{n}}$ (correspondingly, for the pair \( (2n-1, d_{2n-1}) \), the locus $Z_{2n-1,d_{2n-1}}$ has codimension at least 1 in $ \left(\mathbb{C}\mathbb{P}^{2n-1}\times U_{2n-1,d_{2n-1}}\right)\setminus\mathcal{X}_{2n-1,d_{2n-1}}$). Then the locus $Z_{n,d_{2n-1}}$ has codimension at least $n$ in $\left(\mathbb{C}\mathbb{P}^{n}\times U_{n,d_{2n-1}}\right)\setminus\mathcal{X}_{n,d_{2n-1}}$. Hence by definition, for a general universal hypersurface $D\in U_{n,d_{2n-1}}$, the set 
	\[\bs	\big(
	E_{k,m}T_{\mathbb{C}\mathbb{P}^n}^*(\log D)
	\otimes
	\mathcal{O}_{\mathbb{C}\mathbb{P}^n}(1)^{-\widetilde{m}}
	\big)\cap \left(\mathbb{C}\mathbb{P}^n\setminus D\right)\]
	has dimension at most 0 in $\mathbb{C}\mathbb{P}^n$.
	
	To ensure the existence of such a pair $(n, d_{2n-1})$ for $k=n$ and for certain $m,\widetilde{m}$, we follow the techniques and notations in \cite[Proposition 3.1 and pp. 43]{Merker-Ta22}. The lower bound of the degree $d_{2n-1}$ is determined by the largest root of the polynomial equation 
\[
	d^{2n-1}I_0+ d^{2n-2}I_1+\cdots+I_{2n-1}=0,
\]
where $I_0>0$ and $I_q$ $(q=0,\ldots,2n-1)$ are dependent on the given $m,\widetilde{m}$. Applying Fujiwara's bound, it induces that
	\begin{align*}
		\max|\text{roots}|&\leq  2\max\limits_{1\leq p\leq 2n-1}\sqrt[p]{\frac{|I_p|}{I_0}}.
	\end{align*}
	Keeping the estimates in \cite[pp. 12]{Merker-Ta22} and using the improved bound for the pole order of slanted vector fields from \cite{Darondeau_vectorfield}, which is \( 5n-2 \), we choose parameters as
	\[
	\quad c = 2p + 5(2n-1) - 1, \quad l = \sqrt{2n-1}\, \frac{\log\log(2n-1)}{2}.
	\]
	Accordingly, the twisted order $\widetilde{m}\geq cm-(5(2n-1)-2)m=\left(2p+1\right){m}>2pm$. The quantity $\frac{c+2}{2}$, appearing as $\frac{5n +1}{2}$ in \cite{Darondeau2016, Merker-Ta22},  is replaced here by $p + 5n -2$. Then we derive from \cite[pp. 12 \& 43]{Merker-Ta22} that
	\begin{align*}
		\max|\text{roots}|&\leq 8(2n-1)\left(\frac{c+2}{2}\right)(l+3)^{2n-1}\\
		&\leq 16n(p + 5n -2)\left(\sqrt{2n}\, \frac{\log\log(2n)}{2}+3\right)^{2n}\\
		&\leq \frac{16n(p + 5n -2)}{80n^2+8n}\cdot \left[100n^2\left(\frac{\log\log(2n)}{2}+3\right)^{2n}\right]\\
		&\leq \left(1+\frac{2p-5}{10n+1}\right)\left(n\log n\right)^n \cdot C_n,
	\end{align*}
	where $C_n\le e^{2n \log \log \log (2n) - n \log \log n - n \log 2+\frac{2\sqrt{n}3\sqrt{2}}{\log \log (2n)}}  100n^2 <1$ whenever $n\geq 40$ by \cite[pp. 43]{Merker-Ta22}. Put
	\begin{equation*}
		d_{n,p}=
		\begin{cases}
			16n(p + 5n -2)\left(\sqrt{2n}\, \frac{\log\log(2n)}{2}+3\right)^{2n},\quad \text{for }2\leq n\leq 39;\\
			\left(1+\frac{2p-5}{10n+1}\right)\left(n\log n\right)^n,\qquad \qquad\qquad\qquad\,\,\text{whenever $n\geq 40$}.
		\end{cases}
	\end{equation*}
	Then for a generic hypersurface $D\subset\mathbb{CP}^n$ with degree $d\geq d_{n,p}$, we obtain the desired property.
	
	For the second statement, since the subset $B$ has at most 0 dimension and $\mathbb{C}\mathbb{P}^n$ is compact, it must either be empty or consist of finitely many points. Because $f$ is not constant and the image $f(\Delta)\not\subset D$ is connected, we obtain the desired conclusion. 
\end{proof}

\section{Non-integrated defect relation from logarithmic jet differentials}
In this section, we investigate the non-integrated defect for holomorphic maps from an open Riemann surface $M$. Let $f\colon M\rightarrow\mathbb{CP}^n$ be a holomorphic map with reduced representation $f=[f_0:\dots:f_n]$. Let $D\subset \mathbb{CP}^n$ be a hypersurface of degree $d$ such that $f(M)\not\subset D$. For any $p\in M$, let $\nu^f(D)(p)$ denote the intersection multiplicity of $f(M)$ and $D$ at the point $f(p)$. We recall the definition of the non-integrated defect introduced by Fujimoto \cite{Fujimoto83_1}.
\begin{defi}\label{def:non-integrated defect}
	For an arbitrary positive integer $\mu_0$, define
	\[
	\alpha_{\mu_0} ^f(D)
	:=\inf\{\eta\geq 0:\,\eta\,\text{satisfies the condition}\,(\star)\},
	\]
	where the condition $(\star)$ means that there exists a subharmonic function $u\colon M\to [-\infty,\infty)$, which is harmonic on the set $M\setminus f^{-1}(D)$ and satisfies the following two conditions:
	\begin{enumerate}
		\item[(C1).] The function $e^u$ is of class $C^{\infty}$ and $e^u\leq \|f\|^{d\eta}$, where
		\[\|f\|\coloneq\sqrt{\|f_0\|^2+\cdots+\|f_n\|^2}.\] 
		\item[(C2).] For each point $\xi\in f^{-1}(D)$ and for a holomorphic local coordinate $z$ around $\xi$, the function
		\[
		u(z)-\min\{\nu^f(D)(\xi),\mu_0\}
		\log|z-\xi|\quad\in[-\infty,\infty)
		\]
		is subharmonic.
	\end{enumerate}
	The {\sl non-integrated defect} of $f$ with respect to $D$ cut by $\mu_0$ is then defined by
	\[
	\delta_{\mu_0}^f(D):=1-\alpha_{\mu_0} ^f(D).
	\]
\end{defi}
Note that our definition of condition $(\star)$  differs from the one used in \cite{Cai-Ru-Yang25}, but the two formulations are in fact equivalent (see \cite[p. 175]{Ru2023}).
\begin{pro}[\cite{Fujimoto83_1}]
	\label{properties of non-integrated defect}
	The {non-integrated defect} enjoys the following properties:
	\begin{itemize}
		\item[(1)]	$0\leq \delta_{\mu_0}^f(D)\leq \delta_{f,D}^{[\mu_0]}\leq 1$, where $\delta_{f,D}^{[\mu_0]}$ is the usual $\mu_0$-truncated defect  defined as \eqref{def-truncated (usual) defect}.
		\item[(2)] The {non-integrated defect} $\delta_{\mu_0}^f(D)=1$ if $f(M)\cap D=\emptyset$, or more generally, if there is a bounded holomorphic function $g$ on $M$ with zeros of order at least $\nu^f(D)(\xi)$ at each point $\xi\in f^{-1}(D)$.
		\item[(3)] If there exists a positive integer $\mu\geq\mu_0$ such that $f$ is ramified over $D$ with multiplicity at least $\mu$, then
		\begin{equation}\label{non-integrated vs ramification index}
			\delta_{\mu_0}^f(D)\geq 1-\dfrac{\mu_0}{\mu}.
		\end{equation}
	\end{itemize}
\end{pro}

In a more general way, Fujimoto \cite{Fujimoto83_1} introduced the following $C_p^*$-condition.
\begin{defi}
	For a positive integer $p$, the holomorphic map $f\colon M\rightarrow\mathbb{CP}^n$ satisfies the $C_p^*$-condition if there exists a subharmonic function $u$ on $M$ such that $e^u$ is of class $C^{\infty}$ and
	\[
	\lambda e^u\leq \|f\|^{p},
	\]
	where $\lambda$ is a positive real-valued function on $M$ with $\mathrm{d}s^2=\lambda^2|\mathrm{d}z|^2$ for a global holomorphic function $z$ on $M$ with $\mathrm{d}z\not=0$.
\end{defi}
\begin{rem}
	In the case where $x=(x_1,\dots,x_m)\colon M\rightarrow\mathbb{R}^m$ is a minimal surface immersed in $\mathbb{R}^m$, the generalized Gauss map is given by $f=[f_0:\dots:f_{m-1}]\colon M\to \mathbb{CP}^{m-1}$. Then the metric $\mathrm{d}s^2$ on $M$ induced from the standard metric on $\mathbb{R}^m$ is given by
	\[\mathrm{d}s^2=2\|f\|^2|\mathrm{d}z|^2.\]
	By definition, the map $f$ satisfies the $C_{1}^*$-condition. Similarly, concerning the conformal metric $\mathrm{d}s^2=2\|f\|^{2p}|\mathrm{d}z|^2$, such $f$ satisfies the $C_{p}^*$-condition.
\end{rem}

\begin{thm}\label{defect gauss map}
	Let $M$ be an open Riemann surface which is complete with respect to a conformal metric $\mathrm{d}s^2$.
	Let $D\subset \mathbb{CP}^n$ be a hypersurface of degree $d$. Suppose that there exists sufficiently large $m\gg 1$ and $\tilde{m}>2pm$ such that the subset
	\[ B\coloneq
	\bs	\big(
	E_{n,m}T_{\mathbb{C}\mathbb{P}^n}^*(\log D)
	\otimes
	\mathcal{O}_{\mathbb{C}\mathbb{P}^n}(1)^{-\widetilde{m}}
	\big)\cap \left(\mathbb{C}\mathbb{P}^n\setminus D\right)
	\]
	has at most zero dimension. Let $f\colon M\rightarrow \mathbb{CP}^n$ be a non-constant holomorphic map satisfying the $C_{p}^*$-condition and that $f(M)\not\subset D$. Then one has the estimate
	\[
	\delta_1^f(D)\leq 1-\dfrac{1}{d}\bigg(\dfrac{\widetilde{m}}{m}-2p\bigg).
	\]
\end{thm}
\begin{proof}
	Taking the universal covering $\Pi\colon\widetilde{M}\to M$ with the metric $\mathrm{d}\widetilde{s}^2=\Pi^*\left(\mathrm{d}s^2\right)$, it is clear that $(\widetilde{M},\mathrm{d}\widetilde{s}^2)$ is also complete. And the lifting map $\widetilde{f}\coloneq f\circ \Pi\colon \widetilde{M}\to \mathbb{CP}^n$ satisfies the $C_{p}^*$-condition  and $\widetilde{f}(\widetilde{M})\not\subset D$ as well. Therefore, we may assume that $M$ is simply connected. Hence, by the uniformization theorem, $M$ is either conformally equivalent to $\mathbb{C}$ or the unit disc $\Delta$. 
	
	When $M=\mathbb{C}$, or $M=\Delta$ but $\limsup\limits_{r\rightarrow 1}\frac{T_f(r)}{-\log (1-r)}=\infty$, it yields from Corollary~\ref{defect relation} that, for any $p\in \mathbb{Z}_{>0}$, one has
	\[\delta_1^f(D)\leq\delta_{f,D}^{[1]}\leq 1-\dfrac{\widetilde{m}}{md}<1-\dfrac{1}{d}\bigg(\dfrac{\widetilde{m}}{m}-2p\bigg).\]
	Hence we only need to study the case where $M=\Delta$ and $\limsup\limits_{r\rightarrow 1}\frac{T_f(r)}{-\log (1-r)}<\infty$.
	
	Suppose on the contrary that 
	\[
	\delta_1^f(D)> 1-\dfrac{1}{d}\bigg(\dfrac{\widetilde{m}}{m}-2p\bigg).
	\]
	By definition, there exists a non-negative number $\eta<\dfrac{1}{d}\bigg(\dfrac{\widetilde{m}}{m}-2p\bigg)$, and a subharmonic function $u$ on $M$ with $e^u$ of class $\mathcal{C}^\infty$, such that
	\begin{equation}
		\label{e^u vs f}
		e^u\leq \|f\|^{d\eta},
	\end{equation}
	and in a neighborhood of each point $\xi\in f^{-1}(D)$, the function
	\[
	u(z)-
	\log|z-\xi|
	\]
	is subharmonic, where $z$ is a local holomorphic coordinate around $\xi$.
	
	On the other hand, since the set $B$ has at most 0 dimension, by Corollary \ref{cor of ramification gauss maps}, there exists a nontrivial negatively twisted logarithmic $n$-jet differential
	\[
	\mathscr{P}
	\in
	H^0
	\big(
	\mathbb{C}\mathbb{P}^n,
	E_{k,m}T_{\mathbb{C}\mathbb{P}^n}^*(\log D)
	\otimes
	\mathcal{O}_{\mathbb{C}\mathbb{P}^n}(1)^{-\widetilde{m}}
	\big)
	\]
	with $\dfrac{\tilde{m}}{m}-2p>0$, such that $\mathscr{P}\big(j_k(f)\big)\not\equiv0$.
	Let $D_{\mathscr{P},f}$ denote the pole divisor of $\mathscr{P}(j_k(f))$. Then the local expression \eqref{local expression of log jet, second form} of $\mathscr{P}$ yields
	\[
	D_{\mathscr{P},f}\leq m\sum_{z\in \Delta}\min\{\ord_zf^*D,1\}.
	\]
	Consequently, the function
	\begin{equation}
		\label{function v}
		v:=\log \bigg|\dfrac{\mathscr{P}(j_k(f))}{\|f\|^{\tilde{m}}}\bigg|+mu
	\end{equation}
	is subharmonic on $\Delta$, where $u$ is defined in \eqref{e^u vs f}.
	
	Since $f$ satisfies the $C_p^*$-condition, there exists a subharmonic function $w$ on $\Delta$ such that $e^w$ is of class $C^{\infty}$ and
	\begin{equation}
		\label{using Cp condition}
		\lambda e^w\leq \|f\|^p,
	\end{equation}
	where $\lambda$ satisfies $\mathrm{d}s^2=\lambda^2|\mathrm{d}z|^2$. Putting $t=\frac{2p}{\tilde{m}-\eta dm}$, then since $\frac{\tilde{m}}{m}>2p$, one has $0<mt<1$. We then consider the following subharmonic function
	\[
	\varphi:=2w+tv=2w+t\left(\log \bigg|\dfrac{\mathscr{P}(j_k(f))}{\|f\|^{\tilde{m}}}\bigg|+mu\right).
	\]
	Combining \eqref{e^u vs f}, \eqref{function v}, \eqref{using Cp condition} together, one obtains that
	\begin{align*}
		e^{\varphi}\lambda^2
		\leq & \|f\|^{2p}e^{tv}\\
		\leq & \|f\|^{2p}e^{tmu}\bigg(\dfrac{|\mathscr{P}(j_k(f))|}{\|f\|^{\tilde{m}}}\bigg)^t\\
		\leq & \|f\|^{2p+\eta dmt}\bigg(\dfrac{|\mathscr{P}(j_k(f))|}{\|f\|^{\tilde{m}}}\bigg)^t=|\mathscr{P}(j_k(f))|^t.
	\end{align*}
	It follows from \cite[Lemma 2.2]{Cai-Ru-Yang25} that for $0<mt<l<1$ and $0<r_0<r<R<1$, there exists a constant $K_1$ such that
	\begin{align*}
		\int_{0}^{2\pi}e^{\varphi}\lambda^2 (re^{i\theta})d\theta & \leq \int_{0}^{2\pi}|\mathscr{P}(j_k(f))|^t (re^{i\theta})d\theta\\
		& \leq K_1\left(\frac{R}{r(R-r)}T_f(r)\right)^l.
	\end{align*}
	Taking $R=r+\frac{1-r}{e T_f(r)}$, by \cite[Lemma 4.1.7]{Ru2023}, we get $T_f(R)\leq eT_f(r)$ for all $r\in [0,1)$ excluding a set $E$ with $\int_E \frac{dt}{1-t} <\infty$. Thus we deduce that
	\begin{align*}
		\int_{0}^{2\pi}e^{\varphi}\lambda^2 (re^{i\theta})d\theta &\leq K_2\left(\frac{1}{1-r}T_f(r)\right)^l\\
		& \leq K_3 \frac{1}{(1-r)^l}\left(\log \frac{1}{1-r}\right)^l,
	\end{align*}
	for all $r\notin E$, where the last inequality holds since $\limsup\limits_{r\rightarrow 1}\frac{T_f(r)}{-\log (1-r)}<\infty$. After adjusting the parameters appropriately, we conclude from \cite[Proposition 4.1.8]{Ru2023} that, for all $r\in [0,1)$,
	\begin{align*}
		\int_{0}^{2\pi}e^{\varphi}\lambda^2 (re^{i\theta})d\theta &\leq K \frac{1}{(1-r)^l}\left(\log \frac{1}{1-r}\right)^l\\
		&\leq K' \frac{1}{(1-r)^{l'}},
	\end{align*}
	with $l<l'<1$ and suitable positive constants $K,K'$. Hence
	\[
	\iint_{\Delta}\big(e^{\varphi}\lambda^2\big)(re^{i\theta})rdrd\theta\leq \int_{0}^{1} K' \frac{r}{(1-r)^{p'}}dr<\infty,
	\]
	contradicting Yau's estimate \cite{Yau76} that $\iint_{\Delta} e^ud\sigma =\infty$.
\end{proof}
Combining this result and Proposition~\ref{properties of non-integrated defect}, we obtain the following.
\begin{cor}
	\label{ramification gauss map}
	Let $M$ be an open Riemann surface and let $f\colon M\rightarrow\mathbb{CP}^n$ be a holomorphic map. Suppose that $M$ is complete with respect to  the conformal metric given by
	\[
	\mathrm{d}s^2=\|F\|^{2p}|\omega|^2,
	\]
	where $F$ is a reduced representation of $f$ and $\omega$ is a holomorphic $1$-form on $M$. Let $D\subset \mathbb{CP}^n$ be a hypersurface of degree $d$ such that $f(M)\not\subset D$.  Suppose that there exists sufficiently large $m\gg 1$ and $\tilde{m}>2pm$ such that the subset
	\[
	\bs	\big(
	E_{n,m}T_{\mathbb{C}\mathbb{P}^n}^*(\log D)
	\otimes
	\mathcal{O}_{\mathbb{C}\mathbb{P}^n}(1)^{-\widetilde{m}}
	\big)\cap \left(\mathbb{C}\mathbb{P}^n\setminus D\right)
	\]
	has at most zero dimension. Then the following statements hold:
	\begin{enumerate}
		\item $	\delta_1^f(D)\leq 1-\dfrac{1}{d}\bigg(\dfrac{\widetilde{m}}{m}-2p\bigg).$
		\item If $f$ is ramified over $D$ with multiplicity at least $\mu$, then
		\[
		\dfrac{1}{\mu}\geq\dfrac{1}{d}\bigg(\dfrac{\widetilde{m}}{m}-2p\bigg).
		\]
		\item In particular, if the image of $f$ avoids $D$, then $f$ must be constant.
	\end{enumerate}
	
\end{cor}

\section{A curvature estimate}
Theorem~\ref{gauss curvature estimate} can be proved by following the similar arguments as in the proof of \cite[Theorem 2]{Chen-Li-Liu-Ru21}.
Firstly, we consider the normal family of holomorphic maps ramified over a large degree hypersurface with high multiplicity.
\subsection{Normal family of holomorphic maps 
}
Denote by $\hol(\Delta,\mathbb{CP}^n)$ the set of all holomorphic maps from $\Delta$ to $\mathbb{CP}^n$.
\begin{defi}
	Let $\mathcal{F}\subset \hol(\Delta,\mathbb{CP}^n)$ be a family of holomorphic maps. We say $\mathcal{F}$ is normal if it is relatively compact in $\hol(\Delta,\mathbb{CP}^n)$ with respect to the compact-open topology.
\end{defi}
The following criterion of normality for a family of holomorphic maps will be useful in our proof.
\begin{lem}[\cite{ALADRO19911}]\label{lem:normal criterion}
	The family $\mathcal{F}\subset \hol(\Delta,\mathbb{CP}^n)$ is not normal if and only if there exist a compact set $K\Subset \Delta$, and a sequence of points $\{p_i\}_{i\geq 1}\subset K$ with $\{p_i\}\to p_0\in \Delta$, a sequence of holomorphic maps $\{f_i\}_{i\geq 1}\subset \mathcal{F}$, a sequence of parameters $\{\rho_i\}_{i\geq 1}\subset \mathbb{R}_{>0}$ with $\{\rho_i\}\to 0$ such that, the sequence of maps 
	\[g_i(\zeta)\coloneq f_i(p_i+\rho_i \zeta),\quad \zeta\in \mathbb{C},\]
	converges uniformly on compact subsets in $\mathbb{C}$ to a nonconstant holomorphic map $g\colon \mathbb{C}\to \mathbb{CP}^n$.
\end{lem}
\begin{thm}
	\label{normal family}
	Let $D\subset\mathbb{CP}^n$ be a Kobayashi-hyperbolic hypersurface of degree $d$. Let $\mathcal{F}\subset \hol(\Delta,\mathbb{CP}^n)$ be a family of non-constant holomorphic mappings from $\Delta$ into $\mathbb{CP}^n$. Suppose that each $f\in \mathcal{F}$ is ramified over $D$ with multiplicity at least $\mu$, and that there exists a nontrivial negatively twisted logarithmic $k$-jet differential
	\[
	\mathscr{P}
	\in
	H^0
	\big(
	\mathbb{C}\mathbb{P}^n,
	E_{k,m}T_{\mathbb{C}\mathbb{P}^n}^*(\log D)
	\otimes
	\mathcal{O}_{\mathbb{C}\mathbb{P}^n}(1)^{-\widetilde{m}}
	\big)
	\]
	such that
	\begin{align*}
		\mathscr{P}\big(j_k(f)
		\big)
		\not\equiv
		0.
	\end{align*}
	If $\mu>\dfrac{md}{\widetilde{m}}$, then $\mathcal{F}$ is a normal family.
	
\end{thm}
\begin{proof}
	Suppose on the contrary that $\mathcal{F}$ is not normal. Then by Lemma \ref{lem:normal criterion}, there exists a sequence of holomorphic mappings $\{f_i\}_i\subset\mathcal{F}$, a sequence of points $\{p_i\}_i\subset K\Subset \Delta$ with $\{p_i\}_i\to p_0\in \Delta$, a sequence of positive numbers $\{\rho_i\}_i\to 0$ such that the sequence of holomorphic maps
	\begin{align*}
		g_i\colon \Delta_{r_i}&\rightarrow\mathbb{CP}^n\,\quad (r_i\rightarrow\infty),\\
		z&\mapsto f_i(p_i+\rho_iz),
	\end{align*}
	converges uniformly on compact sets of $\mathbb{C}$ to a non-constant holomorphic map $f\colon \mathbb{C}\rightarrow\mathbb{CP}^n$, after passing to a subsequence. By Hurwitz's Theorem, either the image of $f$ is contained in $D$, or $f$ is ramified over $D$ with multiplicity at least $\mu$. The first case contradicts the hyperbolicity of $D$, while the second contradicts Corollary~\ref{ramification}. Therefore, $\mathcal{F}$ is normal on $\Delta$.
\end{proof}
\subsection{Technical lemmas}
We collect some necessary results to prove our main theorem.
\begin{lem}{\cite[Theorem 3]{Chen-Li-Liu-Ru21}}
	\label{curvature estimate, g omits neigh. of hyperplane}
	Let $M$ be an open Riemann surface and let $f\colon M\rightarrow\mathbb{CP}^n$ be a non-constant holomorphic map. Consider the conformal metric on $M$ given by
	\[
	\mathrm{d}s^2=\|F\|^{2p}|\omega|^2,
	\]
	where $F=[f_0\cdots:f_n]$ is a reduced representation of $f$ and $\omega$ is a holomorphic $1$--form on $M$. If $f$ omits a neighborhood $U$ of a hyperplane in $\mathbb{CP}^n$, then there exists a constant $C$ depending only on $U$ such that
	\[
	|K(q)|^{\frac{1}{2}}\mathrm{d}(q)\leq C,
	\]
	where $K(q)$ is the Gauss curvature of the surface at $q$ and $\mathrm{d}(q)$ is the geodesic distance from $q$ to the boundary of $M$.
\end{lem}

\begin{lem}\cite[Lemma 2.1]{Osserman-Ru97}
	\label{compare with the hyperbolic distance} Fix a radius $0<r<1$. Let $R$ denote the hyperbolic radius of the disc $\Delta_r\subsetneq \Delta$. Consider a conformal metric $\mathrm{d}s^2=\lambda^2(z)|\mathrm{d}z|^2$  on $\Delta_r$ such that the geodesic distance from the origin to the boundary $\left\{|z|=r\right\}$ is greater than or equal to $R$. If the Gauss curvature under the metric $\mathrm{d}s^2$ satisfies $-1\leq K\leq 0$, then the distance from any point to the origin with respect to the metric $\mathrm{d}s^2$ is greater than or equal to the hyperbolic distance.
\end{lem}

\begin{lem}\cite[Lemma 2.2]{Osserman-Ru97}
	\label{compare with the hyperbolic distance, after passing limit} 
	Let $\left\{\mathrm{d}s_i^2\right\}_{i\geq 1}$ be a sequence of conformal metrics on the unit disc $\Delta$, whose Gauss curvature satisfies $-1\leq K_i\leq 0$ for all $i$. Assume that $\Delta$ is the geodesic disc of radius $R_i$ with respect to the metric $\mathrm{d}s_i^2$. If $\lim\limits_{i\to\infty} R_i=\infty$ and the sequence $\{\mathrm{d}s_i^2\}_{i\geq 1}$ converges uniformly on any compact set of $\Delta$ to a metric $\mathrm{d}s^2$, then the distance from any point to the origin with respect to the metric $\mathrm{d}s^2$ is greater than or equal to the corresponding hyperbolic distance. In particular, such metric $\mathrm{d}s^2$ is complete.
\end{lem}

\begin{lem}\cite[Proposition 1]{Chen-Li-Liu-Ru21}
	\label{limit of curvatures}
	Let $M$ be an open simply connected Riemann surface and let $\left\{f_i\colon M\rightarrow\mathbb{CP}^n\right\}_{i\geq 1}$ be a sequence of holomorphic mappings. For each $i\geq 1$ and for a positive integer $p$, we consider the sequence of conformal metrics on $M$ given by
	\[
	\mathrm{d}s_i^2=\|F_i\|^{2p}|\mathrm{d}z|^2,
	\]
	where $F_i=[f_{i,0}:f_{i,1}:\dots:f_{i,n}]$ is a fixed reduced representation of $f_i$. Denote by $K_i$ the Gauss curvature with respect to $\mathrm{d}s_i^2$. If the sequence $\{f_i\}_{i\geq 1}$ converges uniformly on every compact subset of $M$ to a non-constant holomorphic map $g$ and the sequence $\{|K_i|\}_{i\geq 1}$ is uniformly bounded, then one of the following statements holds true:
	\begin{enumerate}
		\item There exists a subsequence of $\{K_i\}_{i\geq 1}$ converging to $0$.
		\item For each $0\leq j\leq n$, there exists a subsequence $\{f_{i_k,j}\}_{k\geq 1}$ of $\{f_{i,j}\}_{i\geq 1}$ converging to a holomorphic function $h_j$ on $M$. Moreover, $h_0,\dots,h_n$ have no common zero.
	\end{enumerate}
\end{lem}

\subsection{Proof of Theorem \ref{gauss curvature estimate}}
\begin{proof}[Proof of Theorem~\ref{gauss curvature estimate}]
	Assume on the contrary that the curvature estimate (\ref{eq:curvature-main}) fails. Then there exists a sequence of holomorphic maps $\left\{g_i\colon M_i\rightarrow \mathbb{CP}^n\right\}_{i\geq 1}$, each defined on an open Riemann surface $M_i$ and ramified over $D$ with multiplicity at least $\mu$ satisfying
	\begin{equation}
		\label{multiplicity condition of gi}
		\dfrac{1}{\mu}<\dfrac{1}{d}\bigg(\dfrac{\widetilde{m}}{m}-2p\bigg),
	\end{equation}
	together with a sequence of points $\left\{p_i\in M_i\right\}_{i\geq 1}$ such that
	\[|K_i(p_i)| \mathrm{d}_i^2(p_i)\rightarrow\infty\quad \text{as $i\to\infty$},\]
	with respect to the induced conformal metric $\mathrm{d}s_i^2$. 
	
	For each $i\geq 1$, taking the universal cover of $M_i$ if necessary, we may assume that $M_i$ is simply connected. By the uniformization theorem, $M_i$ is therefore conformally equivalent to either $\mathbb{C}$ or the unit disc $\Delta$. The first case is ruled out by Corollary~\ref{ramification}, since the multiplicity $\mu$ satisfies \eqref{multiplicity condition of gi}. Henceforth we may assume that each $M_i=\Delta$ and $p_i=0$ as the origin. 
	
	Following the same arguments in \cite[~p.465]{Chen-Li-Liu-Ru21}, we can assume furthermore that each Riemann surface $M_i$ can be selected as a geodesic disc of radius $R_i\to \infty$ with respect to the conformal metric $\mathrm{d}s_i^2$, and that
	\[
	K_i(p_i)=-\frac{1}{4}\quad\text{and}\,\, -1\leq K_i\leq 0.
	\]
	Indeed, we can assume $M_i$ to be a geodesic disc centered at $p_i$ and set \[
	M_i^* \coloneq \{p \in M_i : \mathrm{d}_i(p, p_i) \leq \mathrm{d}_i(p_i)/2\}.
	\]
	Then, denoting by \( \mathrm{d}_i^*(p) \) the geodesic distance from \( p \) to \( \partial M_i^* \), the Gauss curvature \(|K_i|\) is uniformly bounded on \( M_i^* \) and \( \mathrm{d}_i^*(p) \) tends to zero as \( p \) tends to \( \partial M_i^* \). Hence there exists an interior point \( p_i^* \in M_i^* \) such that
	\[
	|K_i(p_i^*)|(\mathrm{d}_i^*(p_i^*))^2 = \max_{p \in M_i^*}|K_i(p)|(\mathrm{d}_i^*(p))^2,
	\]
	and thus one has
	\[
	|K_i(p_i^*)|(\mathrm{d}_i^*(p_i^*))^2 \geq |K_i(p_i)|(\mathrm{d}_i^*(p_i))^2 = \frac{1}{4}|K_i(p_i)|(\mathrm{d}_i(p_i))^2 \to \infty.
	\]
	So we can replace each \( M_i \) by \( M_i^* \) with \( |K_i(p_i^*)|\left(\mathrm{d}_i^{*}(p_i^*)\right)^2 \to \infty \) and rescale \( M_i^* \) to make \( K_i(p_i^*) = -\frac{1}{4} \). We keep the notation \( \mathrm{d}_i^*\left(\cdot\right) \) to denote the geodesic distance after rescaling. Again, we can assume \( M_i^* \) to be a geodesic disc centered at \( p_i^* \) and let
	\[
	M_i^{**} \coloneq \{p \in M_i^* : \mathrm{d}_i(p, p_i^*) \leq \mathrm{d}_i^*(p_i^*)/2\}.
	\]
	Then \( p \in M_i^{**} \) implies that \( \mathrm{d}_i^*(p)\geq \left|\mathrm{d}_i^*(p_i^*)- \mathrm{d}_i(p, p_i^*)\right| \geq \frac{\mathrm{d}_i^*(p_i^*)}{2} \) and that
	\[
	|K_i(p)|(\mathrm{d}_i^*(p))^2 \leq |K_i(p_i^*)|(\mathrm{d}_i^*(p_i^*))^2 = \frac{1}{4}(\mathrm{d}_i^*(p_i^*))^2 \leq (\mathrm{d}_i^*(p))^2.
	\]
	That is, \( |K_i(p)| \leq 1 \) for all \( p \in M_i^{**} \). Furthermore, the radius $R_i$ of $M_i^{**}$ satisfies
	\[
	2R_i\geq \mathrm{d}_i(p_i^*, \partial M_i^{**}) = \frac{\mathrm{d}_i^*(p_i^*)}{2} \to \infty. 
	\]
	Replacing each $M_i$ by $M_i^{**}$, we obtain the desired sequence of open Riemann surfaces.
	
	Returning to the sequence of holomorphic maps $\left\{g_i\colon M_i\rightarrow \mathbb{CP}^n\right\}_{i\geq 1}$, it follows from Corollary \ref{cor of ramification gauss maps}, Theorem~\ref{normal family} and the assumption \eqref{multiplicity condition of gi} that the family $\{g_i\}_{i\geq 1}$ is normal. Hence, after extracting a subsequence, we can assume that the family $\{g_i\}_{i\geq 1}$ converges to a holomorphic map $g$ uniformly on any compact subset of $\Delta$. For each $i$, we fix a reduced representation $g_i=[g_{i,0}:g_{i,1}:\dots:g_{i,n}]$. By Lemma \ref{limit of curvatures}, for each $0\leq j\leq n$, there exists a subsequence of holomorphic functions $\{g_{i,j}\}_{i\geq 1}$ converging to a holomorphic function $h_j$ on every compact subset of $\Delta$ and furthermore, the functions $h_0,\dots,h_n$ have no common zero. Obviously, we have $g=[h_0:\dots:h_n]\colon \Delta\to \mathbb{CP}^n$. 
	
	We now claim that $g$ is constant. Indeed, suppose that it is not the case, then by Lemma \ref{compare with the hyperbolic distance, after passing limit} and the previous construction, the metric $\mathrm{d}s^2=\|g\|^{2p}|\mathrm{d}z|^2$ is complete on $\Delta$. Furthermore, by Hurwitz's Theorem, either $g$ is ramified over $D$ with multiplicity at least $\mu$ satisfying \eqref{multiplicity condition of gi} or the image of $g$ must be contained in $D$. This contradicts Corollary \ref{ramification gauss map} and the hyperbolic assumption of $D$.
	
	Now, suppose that $g(\Delta)=\{Q\}\in \mathbb{CP}^n$. Take a hyperplane $H$ which does not contain the point $Q$ and choose two disjoint neighborhoods $U_1$, $U_2$ of $H$, $Q$ respectively. From Lemma~\ref{curvature estimate, g omits neigh. of hyperplane}, for $g\colon \Delta\rightarrow \mathbb{CP}^n\setminus U_1$, there exists a constant $C$ depending on $U_1$ such that
	\[
	|K(p)|^{\frac{1}{2}}\mathrm{d}(p)\leq C.
	\]
	Select $0<r<1$ such that the hyperbolic radius $R$ of the disc $\Delta_r$ satisfies $R>2C$.   
	Since the sequence $\{g_i\}_{i\geq 1}$ converges to $g$ uniformly on $\overline{\Delta}_r$, we obtain that for sufficiently large $i\gg 1$, the image $g_i\left(\overline{\Delta}_r\right)\subset U_2$, omitting $U_1$ consequently. This implies that for such sufficiently large $i$, we have
	\[
	|K_i(0)|^{\frac{1}{2}}\mathrm{d}_i(r)\leq C,
	\]
	where $\mathrm{d}_i(r)=\mathrm{d}_i\left(0,\partial \Delta_r\right)$ is the geodesic distance with respect to the metric $\mathrm{d}s_i^2$. Again by our construction that $K_i(0)=-\frac{1}{4}$ for all $i$, we deduce that 
	\begin{equation}\label{estimate di(r), <=}
		\mathrm{d}_i(r)\leq 2C,\quad \text{for $i\gg 1$}.
	\end{equation}
	Recall our construction that $M_i=\Delta$ is a geodesic disc with radius $R_i\to \infty$ under the metric $\mathrm{d}s_i^2$. We can reparameterize $M_i$ as $\Delta_{r_i}=\{w:|w|<r_i\}$, where $r_i$ is chosen such that the disc $\Delta_{r_i}$ also has a hyperbolic radius $R_i$. Then for above selected $r$, the circle $\{z:|z|=r\}$ corresponds to the circle $\{w:|w|=r_ir\}$. Since $R_i\rightarrow \infty$, one has $r_i\rightarrow 1$ as $i\rightarrow\infty$. This implies that the hyperbolic radius of $\left\{|w|=r_ir\right\}$ tends to the one of $\left\{|w|=r\right\}$, which is $R>2C$. On the other hand, since $-1\leq K_i\leq 0$ on $M_i$,  Lemma~\ref{compare with the hyperbolic distance} yields that the distance with respect to the metric $\mathrm{d}s^2_i$ from the origin to any point on the circle $\{z:|z|=r\}$, or equivalently the circle $\{w:|w|=r_ir\}$, is not less than the hyperbolic distance from the origin to any point on $\{w:|w|=r_ir\}$. Hence we have $\mathrm{d}_i(r)\geq R>2C$, contradicting \eqref{estimate di(r), <=}. 
	This finishes the proof of Theorem~\ref{gauss curvature estimate}.
\end{proof}
\begin{proof}[Proof of Corollary \ref{cor of gauss curvature estimate}]
	This follows directly from Theorem \ref{gauss curvature estimate} and Corollary \ref{cor of ramification gauss maps}, where we choose $\widetilde{m}=(2p+1)m$.
\end{proof}
\section*{Acknowledgements}
D.T. Huynh is supported by the Vietnam National Foundation for Science and Technology Development (NAFOSTED) under the grant number 101.04-2025.19. Yunling Chen thanks to the Academy of Mathematics and Systems Science in Beijing for the great research environment.
\bigskip
\addcontentsline{toc}{chapter}{Bibliography}

\bibliographystyle{plain}
\bibliography{article}
\address
\end{document}